\documentclass[runningheads,numbook]{svjour3}
\usepackage{amsmath,amssymb}
\usepackage{cite}
\usepackage[normalem]{ulem}
\usepackage{color}
\usepackage{mathptmx}   
\usepackage[weather,alpine,misc,geometry]{ifsym}
\usepackage{amsfonts}
\usepackage{graphicx}
\usepackage{hyperref}
\hypersetup{
    colorlinks=true,
    linkcolor=blue,     citecolor=red,     urlcolor=blue  }

\newcommand{\al}{\alpha}

\newcommand{\la}{\lambda}
\newcommand{\de}{\delta}
\newcommand{\eps}{\varepsilon}
\newcommand{\bx}{\bar x}

\newcommand {\R} {\mathbb R}
\newcommand {\N} {\mathbb N}

\newcommand {\B} {\mathbb B}

\newcommand {\dom} {{\rm dom}\,}
\newcommand {\epi} {{\rm epi}\,}

\newcommand {\cl} {{\rm cl}\,}
\newcommand {\co} {{\rm co}\,}

\newcommand {\cone} {{\rm cone}\,}

\newcommand {\sd} {\partial}
\newcommand {\Int} {{\rm int}\,}

\newcommand{\vertiii}[1]{\left\vert\kern-0.25ex\left\vert\kern-0.25ex\left\vert #1\right\vert\kern-0.25ex\right\vert\kern-0.25ex\right\vert}
\newcommand{\vertiiiBig}[1]{\Big\vert\kern-0.25ex\Big\vert\kern-0.25ex\Big\vert #1\Big\vert\kern-0.25ex\Big\vert\kern-0.25ex\Big\vert}

\def\nbh{neighbourhood}
\def\es{\emptyset}
\def\lsc{lower semicontinuous}

\def\RHS{right-hand side}

\def\Fr{Fr\'echet}

\newcommand{\norm}[1]{\left\Vert#1\right\Vert}

\newcommand {\diam} {{\rm diam}\,}

\newcommand{\blue}[1]{\textcolor{blue}{#1}}

\newcommand{\red}[1]{\textcolor{red}{#1}}

\newcommand{\olive}[1]{\textcolor{olive}{#1}}

\newcommand{\qdtx}[1]{\quad\mbox{#1}\quad}

\newcounter{mycount}

\newcommand{\AK}[1]{\todo[inline]{AK {#1}}}
\newcommand{\AH}[1]{\todo[inline,color=green!40]{AH {#1}}}


\smartqed
\usepackage{enumitem}
\usepackage{comment}
\setlist[enumerate,1]{label={\rm(\roman*)}}
\setlist[enumerate,2]{label={\rm(\alph*)}}

\newcommand{\PM}[1]{\todo[inline,color=blue!40]{Reviewer: {#1}}}
\newcommand{\ML}[1]{\todo[inline,color=blue!40]{Marco: {#1}}}

\newcommand{\sdf}{\partial}

\begin{document}


\title{
Strong duality in infinite convex optimization}

\author{Abderrahim Hantoute
and Alexander Y. Kruger
and Marco A. L\'opez}

\institute{
Abderrahim Hantoute \at
Department of Mathematics, University of Alicante, Spain\\
\email{hantoute@ua.es}, ORCID: 0000-0002-7347-048X
\and
Alexander Y. Kruger (corresponding author) \at
Optimization Research Group,
Faculty of Mathematics and Statistics,
Ton Duc Thang University, Ho Chi Minh City, Vietnam\\ \email{alexanderkruger@tdtu.edu.vn},
ORCID: 0000-0002-7861-7380
\and
Marco A. L\'opez \at
Department of Mathematics, University of Alicante, Spain\\
\email{marco.antonio@ua.es},
ORCID: 0000-0002-0619-9618}
\maketitle

\date{}

\if{\begin{abstract}
This paper deals with a general convex optimization problems with
infinitely many constraints, with Lagrangian duals involving infinitely many
dual variables and infinite sums of functions. Two different dual problems
are introduced with the dual spaces being $\ell ^{1}$ and $\ell
^{\infty }$, respectively. We prove that the Slater condition guarantees the
strong duality, in contrast with the Haar duality, where the Slater condition
does not ensure strong duality even in the case of semi-infinite linear
optimization. We apply extended concepts of uniform lower
semicontinuity of infinite
collections of functions, and establish general duality results via fuzzy multiplier
rules for infinite sums.
\end{abstract}
\fi
\begin{abstract}
We develop a methodology for closing
duality gap and
guaranteeing strong duality in infinite convex optimization. Specifically, we examine two new Lagrangian-type dual formulations involving infinitely many
dual variables and infinite sums of functions. Unlike the classical Haar duality scheme, these
dual problems provide
zero duality gap and are solvable under the standard Slater condition. Then we derive general
optimality conditions/multiplier
rules by applying
{subdifferential}
rules for infinite sums established in \cite{HanKruLop1}.
\end{abstract}




\section{Introduction}\label{sec:introduction}

We consider
{the}
\emph{infinite} convex optimization problem
\begin{equation}
\tag{\rm P}
\label{problemp}
\inf_{f_{t}(x)\leq 0\; (t\in T)}\ f_0(x),
\end{equation}
where $T$ is an arbitrary nonempty set with $0\notin T$,
and $f_{t}:X\rightarrow \mathbb{R}_{\infty}:=\R\cup\{+\infty\}$ $(t\in T\cup\{0\})$ are proper convex functions on a
{Banach} space $X$.

If $\dim X<+\infty$, and the functions $f_{t}$ $(t\in T\cup\{0\})$ are linear,
this is the linear \emph{semi-infinite} optimization problem.
Such problems and their duals have important applications in functional approximation, engineering, economy, etc.; see, e.g., \cite{GobLop98} and the references therein.




Dual problems can be associated with \eqref{problemp}
in the context of the
duality theory developed by Duffin \cite{Duf56}, Kretschmer \cite{Kre61} and Borwein \cite{Bor83}.
The reader can find a detailed account of this theory for the linear model in \cite[Chapter~3]{AndNas87}.
Assuming that $T$ is finite, i.e., $T:=\{1,\ldots,k\}$ for some $k\in\N$, the
classical Lagrangian framework yields a natural candidate for the dual problem:
\begin{equation*}
\sup_{\lambda\in \mathbb{R}_+^k}\ \inf_{x\in X}\ L(x,\lambda),
\end{equation*}
where
$$
L(x,\lambda):=f_{0}(x)+\sum_{i=1}^{k}\lambda _{i}f_{i}(x)\quad (\lambda\in \mathbb{R}_+^k,\ x\in X)
$$
with the convention $0\cdot(+\infty)=+\infty$.
Under mild qualification conditions (typically, the \emph{Slater} condition), the dual problem attains its maximum and exhibits zero duality gap with the primal problem \eqref{problemp}, i.e., we have strong duality (see, e.g., \cite[Chapter~4]{CorHanLop23} and \cite[Chapter~2]{Zal02}).
This result also extends to {some} compact continuous semi-infinite optimization problems, specifically when $X=\mathbb{R}^{n},$ the index set $T$ is compact and the
mappings $t\mapsto f_{t}(x)$ $(x\in X)$ are upper semi-continuous.
For a discussion of this class of problems we refer the readers to \cite{GobVol22,DinGobLopVol23}.
More precisely, in such a framework strong duality also holds provided the so-called \emph{Haar dual} is used \cite[Chapter~8]{GobLop98}:
\begin{equation*}
\sup_{\lambda \in \mathbb{R}_+^{(T)}}\inf_{x\in X}\ L(x,\lambda),
\end{equation*}
where $\mathbb{R}_+^{(T)}:=\{\lambda :T\rightarrow \mathbb{R}_+: \lambda (t)=0$
for all but finitely many $t\in T\}$, and $L$ is given by
\begin{gather}
\label{L}
L(x,\lambda)
:=f_{0}(x)+\sum_{t\in T}\lambda _{t}f_{t}(x)\quad (\lambda\in \mathbb{R}_+^{(T)},\ x\in X).
\end{gather}
Hence, $\inf_{x\in X}L(x,\lambda)
=\inf_{x\in X}(f_{0}(x)+\sum_{t\in T,\ \lambda_t>0}\lambda _{t}f_{t}(x)),$ where only finite sums are involved.
However, the Haar dual doest not provide strong duality for general infinite optimization problems, which may not be semi-infinite or compact continuous. This is illustrated by the following classical example due to Karney \cite{Kar83}.
\begin{example}
\label{E1.1}
Let \eqref{problemp} be the following linear semi-infinite optimization problem
in $\mathbb{R}^{2}:$
\begin{equation*}
\inf_{\substack{-u_{2}-1\leq 0,\;
u_{1}\leq 0,\\
\frac{1}{k}u_{1}-u_{2}\leq 0\;(k=3,4,\ldots)}}\ u_2,
\end{equation*}
{Thus, $T=\N$.}
The solution set of this problem is obviously $\R_-\times \{0\}$, and the optimal value is zero.
With $x=(u_1,u_2)\in\R^2$, set $f_0(x):=u_2$, $f_1(x):=-u_2-1$, $f_2(x):=u_1$,
and $f_k(x):=\frac{1}{k}u_{1}-u_{2}$ $(k=3,4,\ldots)$.
Observe that the functions $f_k$ $(k=0,1,\ldots)$ are finite, convex and continuous.
For all $x\in\R^2$ and $\la:=(\la_1,\la_2,\ldots)\in\R^{(\N)}$,
{in view of \eqref{L},} we have
\begin{align}
\notag
L(x,\lambda)=& u_2+\la_1(-u_2-1)+\la_2u_1+\sum_{k=3}^{+\infty} \lambda _{k}\Big(\frac{1}{k}u_{1}-u_{2}\Big)
\\
\label{E1.1-2}
=&\Big(\la_2+\sum_{k=3}^{+\infty} \frac{\lambda_{k}}{k}\Big)u_1 +\Big(1-\la_1-\sum_{k=3}^{+\infty}\lambda_{k}\Big)u_2-\la_1.
\end{align}
{It is easy to check that $\inf_{x\in X} L(x,\la)=-1$ if $\la_1=1$ and $\la_2=\la_3=
\ldots=0$, and $\inf_{x\in X}L(x,\la)=-\infty$ otherwise.}
In other words, the optimal value of the Haar dual is $-1$ and we have a (non-zero) duality gap between problem \eqref{problemp} and its Haar dual, even though the Slater condition is satisfied. 
\sloppy
\end{example}

The above example shows that the framework of the Haar duality is not sufficient to
guarantee strong duality. Several solutions have been proposed to remedy this deficiency of the Haar duality. An approach based on the so-called \emph{Farkas-Minkowski property} is investigated in \cite{GobLop98} (see also \cite[Theorem~8.2.12 and Corollary~8.2.13]{CorHanLop23}).
Instead, the approach adopted in this paper
{relies on weaker}
assumptions (such as the Slater condition) and focuses on broader Lagrangian-type dual formulations. 
To this aim, instead of considering finite sums as in the
{conventional}
duality schemes, we appeal to the concept of infinite (upper) sum defined for any
collection of functions $f_t:X\to \R_\infty$ $(t\in T)$ as (see \cite{HanKruLop1}):
\begin{equation}
\label{f}
\Big(\overline{\sum_{t\in T}}f_t\Big)(x):=
\limsup_{S\uparrow T,\;|S|<\infty}\; \sum_{t\in S}f_t(x) \quad (x\in X),
\end{equation}
where $\limsup$ is the usual upper limit over the family of all finite subsets
of $T$ (directed by ascending inclusions).
{Two new duals are proposed.}

The first one is constructed entirely from the original data of \eqref{problemp}:
\begin{equation}
\label{LagH}
\sup_{\lambda \in \mathbb{R}^T_+}\ \inf_{x\in X}\ L(x,\lambda),
\end{equation}
where $\mathbb{R}_+^T:=\{\lambda :T\rightarrow \mathbb{R}_+\}$ and
\begin{equation}
\label{L+}
L(x,\lambda):=f_{0}(x)+{\overline{\sum_{t\in T}}}\lambda _{t}f_{t}^{+}(x) \quad (\lambda\in \mathbb{R}_+^T,\ x\in X),
\end{equation}
with
$\overline{\sum}_{t\in T}\lambda _{t}f_{t}^{+}$ being the infinite sum of the positive parts of the $f_{t}^+$'s, following \eqref{f} (equivalently, because
{the summands are}
nonnegative,
$\overline{\sum}_{t\in T}\lambda _{t}f_{t}^{+}=\sup_{S\subset T,\;|S|<\infty}\sum_{t\in S}\lambda _{t}f_{t}^{+}$).
Example~\ref{E1.1} also indicates the need for considering the positive part functions rather than the original $f_t$'s.
In fact, employing in \eqref{LagH}
{a} Lagrangian of the form
{\eqref{L+} but with the infinite sum of the $f_t$'s instead of $f_t^+$,}
entails the same duality gap as the Haar dual.

{The second
proposed Lagrangian, given by
$$
L(x,\lambda)
:=f_{0}(x)+{\overline{\sum_{t\in T}}}\lambda _{t}f_{t}(x)+\lambda_\infty f_\infty(x),
$$
involves the infinite sum of the $f_t$'s together with a limit function $f_\infty$ defined in Section~\ref{S5}.
This reflects the
role of compactification arguments
used in our approach.
Although the optimal value of the second dual is
in general not larger than that of the first, we show that these
dual problems are more appropriate for achieving zero duality gap and strong duality properties under the same Slater condition used in the conventional (finite) setting. We also derive general
optimality conditions (multiplier
rules) for problem \eqref{problemp} by using a subdifferential rule from \cite{HanKruLop1}.
We mainly focus on the countable case ($T$ countable).
Our models can be reduced to this setting in a sufficiently general context.




\paragraph{Notation and preliminaries}

Our basic notation is standard; cf., e.g.,
\cite{CorHanLop23, Zal02}.
Throughout the paper, $X$ is a Banach space, $X^*$ is its topological dual, and $\langle\cdot,\cdot\rangle$ is the bilinear form defining the pairing between the two spaces.
We use the same notations
$d(\cdot,\cdot)$ and $\norm{\cdot}$
for distances
(including point-to-set distances)
and norms in all spaces.
Symbols $\R$, $\R_+$ and $\N$ represent the sets of all real numbers, all nonnegative real numbers and all positive integers, respectively,
{and we denote $\R_{\infty}:=\R\cup\{+\infty\}$}.
We make use of the conventions $0\cdot(+\infty)=+\infty$,
$(+\infty)-(+\infty)=+\infty$, $\inf\es_{\R}=+\infty$ and
$\sup\es_{\R_+}=0$, where $\es$ (possibly with a subscript) denotes the empty subset (of a given set).

By $\B$ and $\overline\B$, we denote the \emph{open} and \emph{closed} unit balls in the underlying space, while $B_\de(x)$ and $\overline B_\de(x)$ are the
\emph{open} and \emph{closed}
balls with radius $\de>0$ and centre $x$, respectively.
We write $\B^*$
to denote the \emph{open}
unit ball in the dual to a normed space.
For a subset $U\subset X$,
{its interior and closure are denoted, respectively, by $\Int U$ and $\cl U$, and}
$\diam U:=\sup_{x_1,x_2\in U}d(x_1,x_2)$
{denotes its}
diameter.
{The indicator function of $U$ is defined by $i_U(x)=0$ if $x\in U$ and $i_U(x)=+\infty$ if $x\notin U$.}

For an extended-real-valued function $f\colon X\to\R_{\infty}$,
its domain and epigraph are defined by
$\dom f:=\{x \in X\mid {f(x) < +\infty}\}$
and $\epi f:=\{(x,\mu) \in X\times \R: \mid {f(x) \le  \mu}\}$, respectively. The function $f$ is lower semicontinuous if  $\epi f$ is closed, and convex if $\epi f$ is convex.
The \emph{(Fenchel) subdifferential} of $f$ at $\bar x\in\dom f$ is the (possibly empty) convex set
\begin{gather*}
\sdf f(\bar x):=
\{x^*\in X^*\mid
f(x)-f(\bar x)-\langle x^*,x-\bar x\rangle\geq0
\qdtx{for all}
x\in X\},
\end{gather*}
while, for $\eps\ge0$, the set
\begin{gather*}
\sdf_\eps f(\bar x):=
\{x^*\in X^*\mid
f(x)-f(\bar x)-\langle x^*,x-\bar x\rangle+\eps\geq0
\qdtx{for all}
x\in X\}
\end{gather*}
is referred to as the \emph{$\eps$-subdifferential} of $f$ at $\bar x$.

Given an arbitrary nonempty index set $T$, which can be infinite and even uncountable, we consider the family of all its finite subsets:
\begin{gather}
\label{F(T)}
\mathcal{F}(T):=\{S\subset T\mid {|S|<\infty}\},
\end{gather}
where $|S|$ denotes the \emph{cardinality} (number of elements) of a set $S$. The latter is a \emph{directed set} when endowed with the partial order determined by ascending inclusions.
The upper and lower limits of a net
(of extended real numbers) $\{\alpha_{S}\}_{S\in\mathcal{F}(T)} \subset[-\infty,+\infty]$ are defined, respectively, as usual:
\begin{align}
\label{al}
\limsup_{S\uparrow T,\;|S|<\infty}\al_S
:=\inf_{S_0\in\mathcal{F}(T)} \sup_{S\in\mathcal{F}(T),\;S_0\subset S}\al_S,\quad
\liminf_{S\uparrow T,\;|S|<\infty}\al_S
:=\sup_{S_0\in\mathcal{F}(T)} \inf_{\substack{S\in\mathcal{F}(T),\;S_0\subset S}}\al_S.
\end{align}
If these two limits are equal, the common value, written $\lim_{S\uparrow T}\al_S$, is called the limit of $\{\alpha_{S}\}_{S\in\mathcal{F}(T)}$.
In particular, if $T$ is finite, then $\lim_{S\uparrow T}\al_S=\al_T$. The limit also exists if the mapping $\mathcal{F}(T)\ni S\mapsto\alpha_{S}$
is non-decreasing, i.e., $\alpha_{S_1}\le\alpha_{S_2}$ for all $S_1,S_2\in\mathcal{F}(T)$ with $S_1\subset S_2$.
Since the mappings $S_0\mapsto\sup_{S\in\mathcal{F}(T),\;S\supset S_0}\al_S$ and $S_0\mapsto\inf_{S\in\mathcal{F}(T),\;S\supset S_0}\al_S$ are non-increasing and non-decreasing, respectively, the upper and lower limits above are well defined, and
$$\limsup_{S\uparrow T,\;|S|<\infty}\al_S\ge\liminf_{S\uparrow T,\;|S|<\infty}\al_S.
$$
Given a sequence of real numbers $t_1,t_2,\ldots$, the definitions above reduce to the conventional upper and lower limits, respectively, if we set $\al_S:=t_{\max S}$ for any finite subset $S\subset\N$.

{In Section \ref{S2}, we formulate general subdifferential multiplier rules involving infinite families of convex functions.
Two dual problems to the infinite convex optimization problem~(\ref{problemp}) briefly discussed above are investigated in Section \ref{S5}.
Under the classical Slater condition,
both dual problems
{exhibit} strong duality, ensuring zero duality gap and {the existence of} dual solutions.
In Section \ref{S4}, we derive general
optimality conditions/multiplier
rules for the infinite convex optimization problem~\eqref{problemp}.
Some conclusions are presented in Section \ref{sec:conclusions}.
}

\section{Infinite sums}
\label{S2}

{In this section, we discuss the uniform lower semicontinuity property of the infinite sum function from \cite{HanKruLop1} and formulate general subdifferential multiplier rules.}

Given
{a subset $U\subset X$, a nonempty index set $T$, and}
a collection of extended-real-valued functions $\{f_{t}\}_{t\in T}$
{on $X$},
we consider their (upper) sum \eqref{f}
and
{its}
\emph{uniform infimum}
{on $U$}:
\begin{equation}\label{La0}
{\Lambda}_U(\{f_{t}\}_{t\in T}):=
\liminf_{S\uparrow T,\;|S|<\infty}\;
\liminf\limits_{\substack{\diam\{x_t\}_{t\in S}\to0\\x_t\in U\;(t\in S)}}\; \sum_{t\in S}f_{t}(x_{t}),
\end{equation}
where the $\limsup$ and $\liminf$ operations
over the directed set \eqref{F(T)}
{are defined by \eqref{al}}.
The following result, which shows the preservation of convexity by infinite sums, can be easily verified.

\begin{proposition}
The sum function \eqref{f} is convex whenever all the $f_t$'s are convex.
\end{proposition}

If $\lim_{S\uparrow T,\;|S|<\infty}\; \sum_{t\in S}f_t(x)$ exists
{(in $[-\infty,+\infty]$)}
for all $x\in X$, we write simply $\sum_{t\in T}f_t$ instead of $\overline\sum_{t\in T}f_t$.
Definition \eqref{La0} is an extension of the corresponding definition from \cite{BorZhu96,BorIof96,Las01} to the infinite setting.
One obviously has
\begin{gather}
\label{E2.3}
\Lambda_U(\{f_t\}_{t\in T})\le \inf_U\overline{\sum_{t\in T}}f_{t},
\end{gather}
and the inequality can be strict even when $T$ is finite, as the following example shows (see also \cite[Example~3.10]{FabKruMeh24}).
\begin{example}
If $U_1$, $U_2\subset X$ are such that $U_1\cap U_2=\emptyset$ and $(\cl{U_1})\cap (\cl{U_2})\ne \emptyset$, then
$$0=\Lambda_X(\{i_{U_1}, i_{U_2}\})<  \inf(i_{U_1}+ i_{U_2})=+\infty.$$
\end{example}

The collection $\{f_{t}\}_{t\in T}$ is said to be \emph{uniformly lower semicontinuous}
on $U$ if
\begin{gather}
\label{La0qc}
\inf_U\overline{\sum_{t\in T}}f_{t}\le
\Lambda_U(\{f_t\}_{t\in T}).
\end{gather}
{In view of \eqref{E2.3}, inequality \eqref{La0qc} can only hold as equality.}

{As shown in}
\cite{HanKruLop1},
the collection $\{f_{t}\}_{t\in T}$ is uniformly lower semicontinuous on $U$ if and only if the next two conditions are satisfied:
\begin{gather*}
\limsup_{S\uparrow T,\;|S|<\infty}\Big(\inf_U\sum_{t\in S} f_t- {\Lambda}_U\big(\{f_t\}_{t\in S}\big)\Big)\le0,
\\
\inf_U\overline{\sum_{t\in T}}f_{t}
\le \liminf_{
S\uparrow T, |S|<\infty} \inf_U
\sum_{t\in S}
f_t.
\end{gather*}
{(The latter condition is referred to in \cite{HanKruLop1} as \emph{inf-stability} of the sum.)}
Inequality \eqref{La0qc}
plays the role of a qualification condition.
It lies at the core of the \emph{decoupling approach} (see, e.g., \cite{HanKruLop1} and the references therein).
{Below},
we discuss some situations where uniform lower semicontinuity is satisfied.
Other typical sufficient conditions can be found in \cite{Las01,Pen13,FabKruMeh24,HanKruLop1}.

The following two propositions provide fairly general situations in which the uniform lower semicontinuity holds.
The first one uses the \emph{weak} topology of $X$;
if $\dim X<+\infty$, it obviously coincides with the standard norm topology.

\begin{proposition}
\label{proa}
Suppose $f_{t}$ $(t\in T)$ are weakly lower
semicontinuous, and either $f_{t}$ $(t\in T)$ are
nonnegative or $T$ is finite. Then
$\{f_{t}\}_{t\in T}$ is
uniformly lower semicontinuous on every weakly compact subset of~$X$.
\end{proposition}

\begin{proof}
Let $U$ be a weakly compact subset of $X$.
Suppose $f_{t}$ $(t\in T)$ are
nonnegative.
Choose an arbitrary $S\in\mathcal{F}(T)$.
{Let $\diam\{x_{ti}\}_{t\in S}\to0$.
Choose some $t\in S$.
Since $U$ is weakly compact, we can suppose that
$(x_{ti})$ weakly converges to some $x_S\in U$.
Since $\diam\{x_{ti}\}_{t\in S}\to0$ and the norm is a weakly \lsc\ function, $(x_{ti})$ weakly converge to $x_S$ for all $t\in S$.
The sum $\sum_{t\in S}f_t$ is weakly lower semicontinuous.
Hence,
$${\liminf_i
\sum_{t\in S}f_{t}(x_{ti})}\ge\sum_{t\in S}f_{t}(x_{S}).$$
In view of definition \eqref{La0} and the nonnegativity of the functions, we have
\begin{equation*}
{\Lambda}_U(\{f_{t}\}_{t\in T})\ge
\liminf_{S\uparrow T,\;|S|<\infty}\;\sum_{t\in S}f_{t}(x_S)\ge
\liminf_{S\uparrow T,\;|S|<\infty}\;\sum_{t\in S_0}f_{t}(x_S)
\end{equation*}
for any $S_0\in\mathcal{F}(T)$.
We may assume that the net $(x_S)_{S\in  \mathcal{F}(T)}$ weakly converges to some $\hat{x}\in U$.
Hence,
\begin{equation*}
{\Lambda}_U(\{f_{t}\}_{t\in T})\ge
\sum_{t\in S_0}f_{t}(\hat x).
\end{equation*}
Since $S_0\in\mathcal{F}(T)$ is arbitrary and in view of definition \eqref{f}, we have
}
$$
\Lambda_U(\{f_t\}_{t\in T})\ge{\overline{\sum_{t\in T}}}f_{t}(\hat{x})
\ge \inf_{x\in U} {\overline{\sum_{t\in T}}}f_{t},
$$
and the collection $\{f_{t}\}_{t\in T}$ is uniformly lower
semicontinuous on $U.$

{If $T$ is finite, then the above proof remains valid with $S=S_0=T$ independently of the signs of the functions.}
\qed
\end{proof}

\begin{proposition}
\label{P3.4}
Let $f_{1},f_{2}:X\to\R_\infty$ be convex, $\bx\in\dom f_{1}\cap\dom f_2$, and $0\in\sd f_1(\bx)+\sd f_2(\bx)$.
Then the pair $\{f_{1},f_{2}\}$ is uniformly lower semicontinuous on $X$.
\end{proposition}
\begin{proof}
Let $x_{1}^{\ast }\in \partial f_{1}(\bar{x})$,
$x_{2}^{\ast }\in \partial f_{2}(\bar{x})$ and $x_{1}^{\ast }+x_{2}^{\ast}=0$.
Then
\begin{gather*}
f_1(x)\ge f_1(\bar x)+\langle x_1^*,x-\bar x\rangle
\qdtx{for all}
x\in X,\\
f_2(u)\ge f_2(\bar x)+\langle x_2^*,u-\bar x\rangle
\qdtx{for all}
u\in X,
\end{gather*}
and consequently,
\begin{gather*}
f_1(x)+f_2(u)\ge f_1(\bar x)+f_2(\bar x)+\langle x_1^*,x-u\rangle
\qdtx{for all}
x,u\in X.
\end{gather*}
Hence,
\begin{equation*}
\Lambda_X (\{f_{1},f_{2}\})=\liminf_{\substack{\left\Vert x-u\right\Vert\rightarrow0}}(f_1(x)+f_2(u))\ge f_{1}(\bar{x})+f_{2}(\bar{x})\ge\inf(f_{1}+f_{2}),
\end{equation*}
i.e., $\{f_{1},f_{2}\}$ is uniformly lower semicontinuous on $X$.
\qed
\end{proof}
\begin{remark}
The key sufficient condition $0\in\sd f_1(\bx)+\sd f_2(\bx)$ in Proposition~\ref{P3.4} implies that $\bx$ is a minimum of $f_{1}+f_{2}$.
It is satisfied, e.g., when $\bx$ is a minimum of $f_{1}+f_{2}$ and the subdifferential sum rule holds:
\begin{equation}
\label{aprox}
\partial(f_{1}+f_{2})(\bar{x})=\partial f_{1}(\bar{x})+\partial f_{2}(\bar{x}).
\end{equation}
Any qualification condition that guarantees
this sum rule such as the Fenchel-Moreau-Rockafellar assumption $\dom f_1\cap\Int\dom f_2\ne\es$
(see, e.g., \cite[Proposition~4.1.20]{CorHanLop23}), is sufficient for the uniform lower semicontinuity. 
It can be shown using similar arguments that other
subdifferential rules are also sufficient for this property.
For instance, instead of \eqref{aprox} one can use the formula
\begin{equation}
\label{aproxb}
\partial(f_{1}+f_{2})(\bar{x})={\bigcap}_{\varepsilon >0}{\cl}^{\|\cdot\|_*}
\big(\partial_\varepsilon f_{1}(\bar{x})+\partial_\varepsilon f_{2}(\bar{x})\big),
\end{equation}
where ${\cl}^{\|\cdot\|_*}$ denotes the weak$^*$-closure in $X^*$. It is worth observing that \eqref{aproxb} always holds when $f_{1}$, $f_{2}$ are convex proper lower semicontinuous and $X$ is reflexive \cite{HirPhe93}.
{Alternatively, the sufficient conditions in the above proposition are also ensured by sequential subdifferential formulas such as those in \cite{Thi97}: There are sequences $(x_n)_n$, $(y_n)_n\subset X$, $x_n^*\in \partial f_1(x_n)$, $y_n^*\in \partial f_2(y_n)$ such that
$\|x_n^*+y_n^*\|\to 0$ and
$$
\ x_n,\, y_n\to \bx, \ f_1(x_n)-\langle x_n^*,x_n-\bx\rangle\to f_1(\bx),\ f_2(y_n)-\langle y_n^*,y_n-\bx\rangle\to f_2(\bx).
$$
Recall that this formula always holds  when $f_1,$ $f_2$ are convex proper lower semicontinuous and $X$ is reflexive (see, e.g., \cite[Corollary 7.2.7]{CorHanLop23}).}
\end{remark}


The next theorem gives
dual necessary conditions for a point
to be a minimum of the
(infinite) sum
of \lsc\ convex functions
{under the uniform lower semicontinuity assumption}.
It is an adaption of the corresponding result from \cite{HanKruLop1} to the convex setting.

\begin{theorem}
\label{T3.1}
Let $\bx\in\dom\overline\sum_{t\in T}f_{t}$ be a minimum of $\overline\sum_{t\in T}f_{t}$. Suppose there is a number ${\de>0}$ such that
$f_t$ $(t\in T)$ are convex, lower semicontinuous and bounded from below on $B_\de(\bx)$, and $\{f_{t}\}_{t\in T}$ is uniformly lower semicontinuous on $B_\de(\bx)$.
Then,
for any $\eps>0$ and $S_0\in\mathcal{F}(T)$,
there exist an $S\in\mathcal{F}(T)$ and points $x_t\in B_\eps(\bx)$ $(t\in S)$
such that $S_0\subset S$, and
\begin{gather}
\label{T3.1-1}
\sum_{t\in S}\big(f_{t}(x_t)-f_{t}(\bx)\big)\le0,
\\
\notag
0\in\sum_{t\in S}{\sd}f(x_t)+\eps\B^*.
\end{gather}
\end{theorem}

A slightly more
{restrictive} version of the uniform lower semicontinuity has been introduced recently in \cite{FabKruMeh24,HanKruLop1}
{as an analytical counterpart of the
\emph{sequential uniform lower semicontinuity (ULC property)} used in \cite{BorIof96,BorZhu05}}.
The collection $\{f_{t}\}_{t\in T}$ is \emph{firmly uniformly lower semicontinuous} on $U$ if
$$
\Theta_U(\{f_t\}_{t\in T})=0,
$$
where
$$
\Theta_U(\{f_t\}_{t\in T})
:=\limsup_{S\uparrow T,\;|S|<\infty}\; \limsup_{\substack{\diam\{x_t\}_{t\in S}\to0\\ x_t\in\dom f_t\,(t\in S)}}\;
\inf_{x\in U} \max\Big\{\max_{t\in S}d(x,x_t),\overline{\sum_{t\in T}}f_t(x)-\sum_{t\in S}f_{t}(x_{t})\Big\}.
$$
Observe that
\begin{gather*}
\Theta_U(\{f_t\}_{t\in T}) \ge\inf_U\overline{\sum_{t\in T}}f_{t} -{\Lambda}_U(\{f_{t}\}_{t\in T}),
\end{gather*}
and, so, every firmly uniformly lower semicontinuous collection is uniformly lower semicontinuous (see \cite{HanKruLop1}).

The firm uniform lower semicontinuity property has been shown in \cite{HanKruLop1} to be an appropriate
{tool} for fuzzy \Fr\ and Clarke subdifferential
{sum} rules.
In the convex setting, the latter one takes the following form.

\begin{theorem}
\label{T3.2}
Let $\bx\in\dom\overline{\sum}_{t\in T}f_t$ and $x^*\in\sd{\big( \overline{\sum}}_{t\in T}f_t\big)(\bx)$.
Suppose there is a number ${\de>0}$ such that
$f_t$ $(t\in T)$ are convex, lower semicontinuous and bounded from below on $B_\de(\bx)$, and $\{f_{t}\}_{t\in T}$ is firmly uniformly lower semicontinuous on $B_\de(\bx)$. Then,
for any $\eps>0$ and $S_0\in\mathcal{F}(T)$,
there exist an $S\in\mathcal{F}(T)$ and points $x_t\in B_\eps(\bx)$ $(t\in S)$
such that $S_0\subset S$, and
\begin{gather}
\label{T3.2-1}
\sum_{t\in S}\big(f_{t}(x_t)-f_{t}(\bx)\big)< \varepsilon,
\\
\notag
x^*\in\sum_{t\in S}{\sd}f(x_t)+\eps\B^*.
\end{gather}
\end{theorem}

\begin{remark}
\label{R3.1}
If $T$ is finite, conditions \eqref{T3.1-1} and \eqref{T3.2-1} can be replaced by the more traditional ones:
\begin{gather}
\label{R3.1-1}
|f_t(x_t) - f_t(\bx)|<\varepsilon\quad  (t\in T).
\end{gather}
{Indeed, if $T$ is finite, thanks to the lower semicontinuity of the functions, taking a smaller $\de>0$, one can ensure that $f_t(\bx) - f_t(x_t)<\varepsilon$ for all $t\in T$.
This together with each of the conditions \eqref{T3.1-1} and \eqref{T3.2-1} (with $S=T$ thanks to $T$ being finite) implies \eqref{R3.1-1} (with possibly a different $\eps>0$).}
Thus, taking into account Proposition \ref{proa},
{Theorem~\ref{T3.2}}
recovers the standard fuzzy convex subdifferential rules for finite sums of convex lower semicontinuous functions in reflexive Banach spaces (see, e.g., \cite{Pen96,Thi97}).
\end{remark}

\section{Duality in infinite convex optimization}
\label{S5}

\if{
\PM{
``To illustrate the value of the semiinfinite problem (P), [8,11] do
not provide enough justification from my point of view. These papers have,
altogether, five citations. Can you comment on reasonable applications of
this model? You may cite textbooks dealing with semiinfinite optimization
here.''}
}\fi

In this section, we investigate two dual problems to the infinite convex optimization problem~(\ref{problemp}).
Under the classical Slater condition,
both dual problems
{exhibit} strong duality, ensuring zero duality gap and
{the existence of} dual solutions.

Thanks to the next \emph{reduction lemma},
when $X$ is separable,
the index set $T$ in \eqref{problemp} can be assumed at most countable.

\begin{lemma}
\label{reduction}
Let $X$ be separable, $f_{t}$ $(t\in T)$ be \lsc.
If $T$ is uncountable, then there exists a countable subset $D\subset T$ such that
{$\sup_{t\in T}f_{t}=\sup_{t\in D}f_{t}$}.
\end{lemma}

\begin{proof}
Let $T$ be uncountable.
Fix
an arbitrary rational number $r$.
Denote $f:=\sup_{t\in T}f_{t}$, $U:=\{x\in X\mid f(x)>r\}$ and $U_t:=\{x\in X\mid f_t(x)>r\}$ $(t\in T)$.
If $f(x)>r$
for some $x\in X$, then there is a $t\in T$ such that $f_t(x)>r$.
Hence,
{$U\subset\bigcup_{t\in T}U_t$}.
For each $t\in T$, the set
{$U_t$}.
is open since $f_t$ is \lsc, while the (open, hence, separable) set
{$U$}
endowed with the induced metric topology of $X$ is Lindel\"{o}f
(i.e., from any open cover of $U$ we
can extract a countable subcover \cite{Mun00}).
Thus, we can find a countable subset
$D_{r} \subset T$ such that
$U\subset\bigcup_{t\in{D_{r}}}U_t \subset\bigcup_{t\in{D}}U_t$,
where $D$ is the union of
$D_{r}$ over all rational $r$.
Thus, $D$ is countable and, for any $x\in X$ and rational $r<f(x)$, we have $\sup_{t\in D}f_t(x)>r$.
The density of the rational numbers in $\mathbb{R}$ entails $f(x)\le\sup_{t\in D}f_t(x)$.
Since the opposite inequality holds trivially, and $x\in X$ is arbitrary, the proof is complete.
\qed\end{proof}

In the rest of the section, we assume, for the sake of simplicity, that $T=\N$, the functions $f_k$ $(k\in\N)$ are convex,
{and $C:=\dom f_0\cap\{x\in X\mid f_k(x)\le 0 \ \forall k\in\N\}\ne\es$}.

Along with problem \eqref{problemp} which takes the form
\begin{equation*}
\inf_{f_{k}(x)\leq 0\ (k\in\N)}\ f_0(x),
\end{equation*}
we consider the family of
perturbed problems depending on $\varepsilon \in \mathbb{R}$:
\begin{equation}
\label{pepsilon}
\tag{\rm P$_\eps$}
\inf_{f_{k}(x)\leq \eps\ (k\in\N)}\ f_0(x),
\end{equation}
Let $v:\mathbb{R}\rightarrow[-\infty,\infty)$ denote the (convex) \emph{value function} which
assigns to each $\varepsilon \in \mathbb{R}$ the optimal value of the perturbed problem \eqref{pepsilon}.
Then $v(\mathtt{P})=v(0)$, where $v(\mathtt{P})$ stands for the optimal value of problem \eqref{problemp}, and,
for all $\eps\ge0$ and $x\in C$,
we have $v(\varepsilon )\leq v(0)\leq
f_{0}(x)<\infty $;
hence, $[0,+\infty)\subset \dom v.$

Two types of \emph{Lagrangians} are defined
for all $x\in X$, $\la_{\infty},\la_1,\la_2,\ldots\in
{\R_+}$,
{and any $m=0,1\ldots$}
as follows:
\begin{gather}
\label{Lag}
L(x,\lambda,\lambda_{\infty}):= f_{0}(x)+\overline{\sum_{k\in\N}}\lambda_{k}f_{k}(x) +\lambda_{\infty}f_{\infty}(x),
\\
\label{Lag2}
\displaystyle L_{m}(x,\lambda):=\displaystyle f_{0}(x)
+\sum_{k=1}^m\lambda_{k}f_{k}(x) +
\sum_{k=m+1}^{+\infty}
\lambda_{k}f_{k}^{+}(x),
\end{gather}
\if{
\AK{31/07/24.
To Marco: I write the summation indices this way to save the vertical space; after this section is merged with the main part, the manuscript can become too long.
The sign $\sum$ looks a bit too big in the article style; this may resolve automatically after the merger; then we will be able to restore the conventional positions of the indices.}
}\fi
where $\la:=(\la_1,\la_2,\ldots)$,
$f_{k}^{+}:=\max \{f_{k},0\}$ is the positive part of $f_{k}$, and
\begin{align}
\label{infty}
f_{\infty }(x):=\limsup_{k\rightarrow\infty}f_{k}(x)
\quad (x\in X).
\end{align}
By convention, the first sum in \eqref{Lag2} equals zero when $m=0$,
and the infinite sums are defined in accordance with the standard definition
\eqref{f}
adopted in this paper:
for any
{$x\in X$ and}
$j\in\N$,
\begin{align}
\overline{\sum_{k>j}}\lambda
_{k}f_{k}(x)&:=\limsup_{n\rightarrow +\infty }\sum_{k=
{j+1}
}^n\lambda
_{k}f_{k}(x).
\label{deflimsup}
\end{align}

Employing Lagrangians \eqref{Lag} and \eqref{Lag2}, one can associate with \eqref{problemp} two dual problems:
\begin{gather}
\tag{\rm D}
\label{problemd}
\sup_{\lambda \in \ell^{1}_{+},\; \lambda_{\infty }\in \mathbb{R}_{+}}\ \inf_{x\in X}\ L(x,\lambda ,\lambda_{\infty }),
\\
\tag{\rm D$_m$}
\label{problemdmas}
\sup_{\lambda \in \ell^{\infty}_{+}}\ \inf_{x\in X}\ L_m(x,\lambda),
\end{gather}
where $\ell^{1}_{+}$ and $\ell^{\infty}_{+}$ denote, respectively, the subsets of $\ell^{1}$ and $\ell^{\infty}$ (the usual $\ell^{p}$ spaces with $p=1$ and $p=\infty$, respectively), whose members have nonnegative elements.
Let $v(\mathtt{D})$ and $v(\mathtt{D}_m)$ stand for the optimal values of
problems \eqref{problemd}
and \eqref{problemdmas},
respectively.

In the framework of the conventional convex optimization (with $m$ constraints), problems \eqref{problemd} and \eqref{problemdmas} reduce
to the usual dual problem:
%
\begin{gather*}
\sup_{\lambda=(\la_1,\ldots,\la_m) \in \R^{m}_{+}}\ \inf_{x\in X}\ \Big(f_{0}(x)
 +\sum\limits_{k=1}^m\lambda_{k}f_{k}(x)\Big).
\end{gather*}
{Indeed, it suffices to set}
$f_{m+1}\equiv f_{m+2}\equiv\ldots{\equiv-1}$.
{Then $f_\infty \equiv-1$ (see \eqref{infty})} and,
{in view of \eqref{Lag},}
problem \eqref{problemd}
{immediately gives $\la_\infty=0$}.
In this ordinary framework, we always have the
equality $(\cl v)(0)=v(\mathtt{D})$,
where $\cl v$ is the \lsc\ envelope of $v$
(see, e.g., \cite[Theorem 2.6.1\,(iii)]{Zal02}). Moreover, provided that the \emph{Slater condition}
\begin{equation}
\label{Sl}
\Big\{x\in\dom f_{0}\mid
{\sup_{k\in\N}f_{k}(x)}<0\Big\}\neq \emptyset
\end{equation}
is satisfied,
strong duality always holds in the ordinary convex optimization: $v(\mathtt{P})= v(\mathtt{D})$ and the value $v(\mathtt{D})$ is attained at some $\la\in\R^{m}_{+}$ (see, e.g.,  \cite[Theorem~2.9.3]{Zal02}).
 Our first aim is to show that in
infinite convex optimization
the main relations between the optimal values are maintained,
and strong duality occurs under
the Slater condition.

We introduce the following sets and recall some of their properties that can be find in  \cite{CorHan24}:
\begin{align*}
\Delta &:=\Big\{(\la_1,\la_2,\ldots)
\in \ell_{+}^{1}\mid \sum_{k=1}^{+\infty}\lambda _{k}\leq
1\Big\} ,
\\
\Delta _{0}&:=\Big\{(\la_1,\la_2,\ldots)
\in \ell_{+}^{1}\mid \sum_{k=1}^{+\infty}\lambda _{k}=1,\;\lambda _{k}=0\text{ for all but finitely many }k\in\N\Big\} .
\end{align*}
The dual space of  $\ell ^{\infty }$ can be expressed as the direct sum $(\ell ^{\infty })^{\ast }=\ell ^{1}\oplus (\ell ^{1})^{\perp }$, i.e., every
$\hat{\lambda}\in (\ell ^{\infty })^{\ast }$ is uniquely
{represented} as the sum of a
$\lambda\in \ell ^{1}\ (\subset (\ell ^{\infty })^{\ast })$ and a
$\mu\in (\ell ^{\infty })^{\ast }$ such that $\mu(\lambda')=0$, for all $\lambda'\in \ell ^{1}$.

\sloppy
We also have
$\Delta _{0}\subset(\ell ^{\infty })^{\ast }$,
and the
closure of $\Delta _{0}$ in the weak$^*$ topology $w^{\ast }:=\sigma ((\ell ^{\infty })^{\ast
},\ell ^{\infty })$ is the $w^{\ast }$-compact set (see \cite[Lemma~15]{CorHan24})
\begin{gather*}
\cl^{w^{\ast}}\Delta_{0}=\Big\{\lambda +\mu \in (\ell
^{\infty })^{\ast }\mid\  \lambda
{=(\la_1,\la_2,\ldots)}
\in \ell_{+}^{1},\;
\mu \in (\ell
^{1})^{\perp },\text{ }\sum_{k=1}^{+\infty}\lambda _{k}+\mu ({\rm e})=1\Big\},
\end{gather*}
where ${\rm e}:=(1,1,\ldots)\in \ell ^{\infty }$.

\quad

The following result relies on the standard minimax theorem (see,
e.g., \cite[Section~3.4]{CorHanLop23}).

\begin{proposition}
\label{lema1}
Let the function $f:=\sup_{k\in\N}f _{k}$ be proper,
{and $f_{\infty}$ be given by \eqref{infty}}.
Then
\begin{equation*}
\inf f=\max_{\substack{(\la_1,\la_2,\ldots)
\in \ell^1_+,\;\la_{\infty}\ge0\\\sum_{k=1}^{+\infty}
\lambda _{k}+\lambda _{\infty }=1}}\;
\inf\Big(\overline{\sum_{k\in\N}}\lambda _{k}f_{k}+\lambda _{\infty }f_{\infty}\Big).
\end{equation*}
\end{proposition}

\begin{proof}
We prove the nontrivial inequality
\textquotedblleft $\leq $\textquotedblright .
For each $x\in X$ and $\hat\lambda\in (\ell^{\infty})^{\ast}$, denote
\begin{gather}
\label{L2P1}
\psi(x,\hat\lambda):= \limsup_{\Delta_{0}\ni(\la_1,\la_2,\ldots) \stackrel{w^*}{\to}\hat\lambda}\; \sum_{k=1}^{+\infty}
\la_{k}f_{k}(x),
\end{gather}
where
``$\stackrel{w^*}{\to}$'' refers to the convergence with respect to the weak$^*$ topology $w^{\ast }:=\sigma ((\ell ^{\infty })^{\ast
},\ell ^{\infty })$ and, thanks to the definition of $\Delta_{0}$, only a finite number of summands in the \RHS\ of \eqref{L2P1} can be nonzero.
Observe that, $\psi (\cdot ,\hat\lambda )$ is convex for each $\hat\lambda \in \cl^{w^{\ast}}\Delta_{0},$ and $\psi (x,\cdot )$ is
concave
and
$w^{\ast }$-{upper semicontinuous}
for each $x\in \dom f .$
The minimax
theorem (see, e.g., \cite
[Theorem~3.4.8]
{CorHanLop23})
gives
the existence of some
${\hat\lambda_0:=(\la_1^\circ,\la_2^\circ,\ldots)}
\in \ell_{+}^{1}$ and ${\mu }\in (\ell ^{1})^{\perp }$ such that $\sum_{k=1}^{+\infty}
{{\lambda}_{k}^\circ}+{\mu }({\rm e})=1$, and
\begin{equation*}
\inf_{x\in \dom f }\sup_{\hat\lambda \in \cl^{w^{\ast}}\Delta_{0}}
\psi (x,\hat\lambda )=\max_{\hat\lambda \in
\cl^{w^{\ast}}\Delta_{0}}
\inf_{x\in \dom f}\psi (x,\hat\lambda) =\inf_{x\in\dom f}\psi(x,
{\hat\lambda_0}
+{\mu}).
\end{equation*}
In view of \eqref{L2P1},
for each $x\in\dom f$, we have (see, e.g., \cite[formula (2.45)]{CorHanLop23}):
\begin{equation*}
f (x)=\sup_{(\la_1,\la_2,\ldots)\in \Delta_{0}}\;
\sum_{k=1}^{+\infty}
\lambda_{k}f_{k}(x)\leq
\sup_{\hat\la\in\cl^{w^{\ast}}\Delta_{0}}
\psi(x,\hat\lambda), 
\end{equation*}
and consequently,
\begin{equation}
\label{L2P4}
\inf_{x\in\dom f}f(x)\le\inf_{x\in\dom f}\psi(x,
{\hat\lambda_0}
+{\mu}) =\inf_{x\in\dom f} \limsup_{\Delta_{0}\ni(\la_1,\la_2,\ldots) \stackrel{w^*}{\to}
{\hat\lambda_0}
+{\mu}}\;
\sum_{k=1}^{+\infty}
\la_{k}f_{k}(x).
\end{equation}
For any $x\in\dom f$ and $n\in\N$, we have
\begin{gather*}
\limsup_{\Delta_{0}\ni(\la_1,\la_2,\ldots) \stackrel{w^*}{\to}
{\hat\lambda_0}
+{\mu}}\;
\sum_{k=1}^{+\infty}
\la_{k}f_{k}(x)
\le\sum_{k=1}^n
{\lambda_{k}^\circ}
f_{k}(x)+ \Big(
\sum_{k=n+1}^{+\infty}
{\lambda_{k}^\circ}
\Big)f(x)+{\mu}({\rm e})\sup_{k> n}f _{k}(x).
\end{gather*}
Taking limits as $n\rightarrow\infty$, we have $\sum_{k=n+1}^{+\infty}{\lambda}_{k}\to0$, and consequently,
\begin{equation}
\label{L2P6}
\limsup_{\Delta_{0}\ni(\la_1,\la_2,\ldots) \stackrel{w^*}{\to}
{\hat\lambda_0}
+{\mu}}\;
\sum_{k=1}^{+\infty}
\la_{k}f_{k}(x)
\leq\overline{\sum_{k\in\N}}
{\hat\lambda_{k}^\circ}f _{k}(x)+{\mu}({\rm e})f _{\infty}(x).
\end{equation}
Combining \eqref{L2P4} and \eqref{L2P6}, and
denoting ${\lambda }_{\infty }:={\mu }({\rm e})$, we obtain:
\begin{equation*}
\inf_X f=\inf_{\dom f }f\leq
\inf_{\dom f }\Big( \overline{\sum_{k\in\N}}
{\lambda_{k}^\circ}
f _{k}+{\lambda }_{\infty }f _{\infty}\Big).
\end{equation*}
To complete the proof, it suffices to observe that $\dom f =\dom \big(
{\overline\sum_{k\in\N}}
{\lambda_{k}^\circ}
f _{k}+\lambda_{\infty }f _{\infty }\big) $.
\qed\end{proof}
\begin{remark}
\label{R3}
The proof of Proposition~\ref{lema1} does not strongly use the fact that the initial point of the sequence $\{f_k\}$ corresponds to $k=1$.
It can be any integer.
This observation, in particular, leads to the conclusion that,
if the function $f:=\sup_{k\in\N\cup\{0\}}f _{k}$ is proper, then
\begin{equation*}
\inf f=\max_{\substack{(\la_1,\la_2,\ldots)
\in \ell^1_+,\;\la_{0}\ge0,\;\la_{\infty}\ge0\\ \sum_{k=0}^{+\infty}
\lambda _{k}+\lambda _{\infty }=1}}\;
\inf\Big(\la_{0}f_{0}+\overline{\sum_{k\in\N}} \lambda _{k}f_{k}+\lambda _{\infty }f _{\infty }\Big).
\end{equation*}
\end{remark}

The next theorem establishes the relationship between problems
\eqref{problemp}, \eqref{problemd} and \eqref{problemdmas}.

\begin{theorem}
\label{tmain}
Suppose that $v(\mathtt{P})>-\infty$. Then, for any $m=0,1\ldots$, it holds:
\begin{equation}
\label{T4-1}
(\cl v)(0)\leq v(\mathtt{D})\leq v(\mathtt{D}_{m})\leq v(\mathtt{P}).
\end{equation}
If, additionally,
the Slater condition \eqref{Sl} is satisfied,
then strong duality holds: $v(\mathtt{P})=v(\mathtt{D})=v(\mathtt{D}_{m})$, and both \eqref{problemd} and \eqref{problemdmas} have solutions, i.e.,
\begin{align}
\notag
v(\mathtt{P})=&\max_{(\la_1,\la_2,\ldots)\in \ell_{+}^{1},\; \lambda_{\infty}\in\mathbb{R}_{+}}\;
\inf\Big(f_{0}+\overline{\sum_{k\in\N}}\lambda _{k}f_{k}+\lambda _{\infty }f_{\infty}\Big)
\\
\label{T4-2}
=&\sup_{(\la_1,\la_2,\ldots)\in
{\ell^\infty_+
}}\;
\inf\Big(f_{0}+\sum_{k=1}^m\lambda _{k}f_{k}+
\sum_{k=m+1}^{+\infty}
\lambda _{k}f_{k}^+\Big).
\end{align}
Moreover, if $\lambda=(\la_1,\la_2,\ldots)\in\ell^1_+$, $\lambda_\infty\in \R_+$ is a solution of \eqref{problemd}, then
\begin{align}
\label{T4-3}
\hat\la:= (\lambda_1,\ldots, \lambda_{m},
{\lambda_{m+1}+\lambda_\infty+\alpha, \lambda_{m+2}+\lambda_\infty+\alpha},
\ldots)
\end{align}
is a solution of \eqref{problemdmas},
{where $\al=0$ if $\lambda_\infty>0$ and $\al=1$ otherwise}.
\end{theorem}

\if{
\begin{remark}
Observe that,
\red{when $\lambda_\infty>0$, condition \eqref{gh} holds with $\alpha_{m+1}=\alpha_{m+2}=\ldots=0$.}
\end{remark}
}\fi
\begin{proof}
The last inequality in \eqref{T4-1} holds true trivially.
We now prove the second inequality.
Let $x\in X$, $\la=(\la_1,\la_2,\ldots)\in\ell^1_+$, $\la_{\infty}\ge 0$,
{$\al=0$ if $\lambda_\infty>0$ and $\al=1$ otherwise},
and $\hat\la$ be given by \eqref{T4-3}.
Obviously, $\hat\la\in\ell^\infty_+$.
{Let $x\in\dom f_\infty$.
If $\la_{\infty}=0$},
then
\begin{align*}
L(x,\lambda,\lambda_{\infty})
&\le f_{0}(x)
+\sum_{k=1}^m\lambda_{k}f_{k}(x)
+\sum_{k=m+1}^{+\infty}(\lambda_{k}+\alpha
)f_{k}^+(x) =L_m(x,\hat\lambda).
\end{align*}
If $\la_{\infty}>0$,
then, for each $\varepsilon > 0$, there is a
$j > m$ such that $f_{\infty}(x)<f_{j}(x)+\varepsilon/\lambda_\infty$, and consequently,
$\la_{j}f_{j}(x)+\la_{\infty}f_{\infty}(x) {<(\la_{j}+\la_{\infty})f_{j}(x)+\varepsilon}.
$
Thus,
\begin{align*}
L(x,\lambda,\lambda_{\infty})
&<f_{0}(x)+
\sum_{k=1}^m \lambda_{k}f_{k}(x)
{+\overline{\sum}_{k>m,\, k\ne j} \lambda_{k}f_{k}(x)}
+{(\la_{j}+\la_{\infty})f_{j}(x)+\varepsilon}\\
&\leq
f_{0}(x)
+\sum_{k=1}^m \lambda_{k}f_{k}(x)
+\sum_{k=m+1}^{+\infty} (\la_{k}+
\la_{\infty})f_{k}^+(x)+\varepsilon
=L_m(x,\hat\lambda)+\varepsilon.
\end{align*}
Hence,
{letting $\varepsilon\downarrow 0$, we obtain
$L(x,\lambda,\lambda_{\infty})\le L_m(x,\hat\lambda)$}.
If $x\notin\dom f_\infty$, then $\limsup_{k\to+\infty}f_k(x)=+\infty$, and consequently,
\begin{align*}
L_m(x,\hat\lambda)=f_{0}(x)+\sum_{k=1} ^m\lambda_{k}f_{k}(x)+\sum_{k=m+1}^{+\infty} (\lambda_{k}+\alpha
+\la_{\infty})f_{k}^+(x)=+\infty.
\end{align*}
Thus, in all three cases we have $L(x,\lambda,\lambda_{\infty})\le L_m(x,\hat\lambda)$.
It follows that $v(\mathtt{D})\leq v(\mathtt{D}_{m})$.

We now
proceed to showing the inequality $v(\mathtt{P})\le v(\mathtt{D})$
under condition (\ref{Sl}).
Assume without loss of generality
that $v(\mathtt{P})=0.$
Thus, $\inf_{X}\sup_{k\in\N\cup\{0\}}f_{k}=0,$
the function $f:=\sup_{k\in\N\cup\{0\}}f _{k}$ is proper, and applying
Proposition~\ref{lema1}
together with Remark~\ref{R3},
we find some
$\lambda =(\la_1,\la_2,\ldots)\in \ell_{+}^{1}$, ${\lambda }_{0}\geq 0$ and ${\lambda }_{\infty}\geq 0$ such that ${\sum_{k=0}^{+\infty}
\lambda _{k}+\lambda _{\infty }=1}$
and
\sloppy
\begin{equation*}
\inf\Big( {\lambda }_{0}f_{0}+ \overline{\sum_{k\in\N}}~{\lambda}_{k}f_{k}+{\lambda}_{\infty}f_{\infty}\Big) =0.  
\end{equation*}
In view of condition (\ref{Sl}), we have $\lambda_{0}>0$ and, so, dividing by $\lambda_{0}$, we obtain
\begin{equation*}
\inf\Big( f_{0}+\overline{\sum_{k\in\N}}~{\lambda }_{0}^{-1}{\lambda}_{k}f_{k}+{\lambda }_{0}^{-1}{\lambda}_{\infty}f_{\infty}\Big) =0.  
\end{equation*}
Hence, $v(\mathtt{D})\ge0=v(\mathtt{P})$, and consequently, the last two inequalities in \eqref{T4-1} hold as equalities,
{and both \eqref{problemd} and \eqref{problemdmas} have solutions}.
The representations in \eqref{T4-2} follow from the definitions, while the last assertion of the theorem is a consequence of the estimates established above.

To establish the first inequality in \eqref{T4-1}
without assumption \eqref{Sl},
we fix {an}
$\varepsilon >0$ and apply the
fact established
above to the perturbed problem
(\ref{pepsilon}), which obviously
satisfies
the Slater condition
({recall}
that the original problem
\eqref{problemp}
is assumed
consistent).
We find some ${\lambda }^{\varepsilon }
=(\la_1^\eps,\la_2^\eps,\ldots)
\in \ell_{+}^{1}$ and ${\lambda }_{\infty
}^{\varepsilon }\geq 0$ such that
\begin{equation*}
v(\varepsilon )=v(\mathtt{P}_{\varepsilon })
=\inf\Big(f_{0}+
\overline{\sum_{k\in\N}} {\lambda}_{k}^{\varepsilon} (f_{k}-\eps)+{\lambda }_{\infty}^{\varepsilon} (f_{\infty}-\eps)\Big)\leq v(\mathtt{D}).
\end{equation*}
Since $v(\varepsilon )=+\infty $ for all $\eps<0$,
we obtain $(\cl v)(0)\leq v(\mathtt{D})$.
\qed\end{proof}

\begin{corollary}
\label{C5}
Let $\bx\in X$ be a solution of problem \eqref{problemp}, and Slater condition \eqref{Sl} hold.
\begin{enumerate}
\item
\label{C5.1}
There exists a solution $(\lambda_{1},\lambda _{2},\ldots )\in \ell_{+}^{1}$, $\lambda _{\infty }\in\R_+$ of problem \eqref{problemd}, i.e.,
\begin{gather}
\label{C5-1}
f_0(\bx)=\inf\Big(f_0
+
{\overline{\sum_{k\in\N}}}
\la_k f_k
+\lambda _{\infty }f_{\infty}\Big).
\end{gather}
Moreover, for any solution of problem \eqref{problemd}, the infimum in \eqref{C5-1} is attained at $\bx$, and the complementary slackness conditions hold:
\begin{gather}
\label{C5-2}
\lambda_{k}f_{k}(\bx)=0\quad (k\in\N\cup \{\infty\}).
\end{gather}
\item
\label{C5.2}
For any $m=0,1,\ldots$,
there exists a solution $(\lambda_{1},\lambda _{2},\ldots )\in \ell_{+}^{\infty}$ of problem \eqref{problemdmas}, i.e.,
\begin{gather}
\label{C5-3}
f_0(\bx)=\inf\Big(f_{0}+\sum_{k=1}^m\lambda _{k}f_{k}+
\sum_{k=m+1}^{+\infty}\lambda _{k}f_{k}^+\Big).
\end{gather}
Moreover, for any solution of problem \eqref{problemdmas}, the infimum in \eqref{C5-3} is attained at $\bx$, and the (partial) complementary slackness conditions hold:
\begin{gather}
\label{C5-4}
\lambda_{k}f_{k}(\bx)=0\quad (k=1,\ldots,m).
\end{gather}
\end{enumerate}
\end{corollary}
\begin{proof}
The first parts of both assertions are direct consequences of representations \eqref{T4-2} in Theorem \ref{tmain}.
Since $\bx\in X$ is a solution of problem \eqref{problemp}, we have $f_k(\bx)\le0$ for all $k\in\N$, and consequently, $f_\infty(\bx)\le0$.
If $(\lambda_{1},\lambda _{2},\ldots )\in \ell_{+}^{1}$, $\lambda _{\infty }\in\R_+$ is a solution of problem \eqref{problemd}, then, by \eqref{C5-1},
\begin{gather*}
f_0(\bx)\leq f_0(\bx)
+\overline{\sum_{k\in\N}}\la_k f_k(\bx)
+\lambda _{\infty }f_{\infty}(\bx),
\end{gather*}
and consequently, conditions \eqref{C5-2} are satisfied, and the infimum in \eqref{C5-1} is attained at $\bx$.
Similarly, if for some $m=0,1,\ldots$,
$(\lambda_{1},\lambda _{2},\ldots )\in \ell_{+}^{\infty}$ is a solution of problem \eqref{problemdmas}, then, by \eqref{C5-3}, taking into account that $f_{k}^+(\bx)=0$ for all $k=m+1,\ldots$,
\begin{gather*}
f_0(\bx)\leq f_0(\bx)
+\sum_{k=1}^m\lambda _{k}f_{k}(\bx),
\end{gather*}
and consequently, conditions \eqref{C5-4} are satisfied, and the infimum in \eqref{C5-3} is attained at~$\bx$.
\qed\end{proof}

\section{Optimality conditions}
\label{S4}

In this
section, we derive general
optimality conditions/multiplier
rules for the infinite convex optimization problem~\eqref{problemp}.
\if{
\AK{24/06/24.
I have replaced application of Theorem 4.2 (sum rule) by that of the main Theorem 4.1 (optimality conditions).
I am explicitly assuming the ``singular'' function $f_\infty$ to be lsc.
(Can its domain not be closed?)
This does not look nice, but seems to simplify everything.
Some sufficient conditions for the lower semicontinuity mentioned in Abderrahim's email can be discussed in a remark.

Have I ruined your theory?}

\AH{22/07/24. We have removed the conclusion on the complementary slackness in the following theorem, since it is already implicit in the definition of the solution $\bx$ when we suppose the existence of a solution\ $\lambda =(\lambda_{1},\lambda _{2},\ldots )\in \ell_{+}^{1}$ and $\lambda _{\infty }\in\R_+$ of problem $(\ref{problemd})$.}
}\fi

\begin{theorem}
\label{optimality}
Let $f_k$ $(k\in\N\cup\{0\})$ be convex, and $\bx\in X$ be a solution of problem \eqref{problemp}.
Suppose that
Slater condition \eqref{Sl} holds,
and there exists a solution\ $\lambda =(\lambda_{1},\lambda _{2},\ldots )\in \ell_{+}^{1}$, $\lambda _{\infty }\in\R_+$ of problem \eqref{problemd} such that
\begin{enumerate}
\item
\label{T5.1}
the functions $f_{0},\;\lambda_{k}f_{k}\,(k\in \mathbb{N}\cup\{\infty\})$ are lower semicontinuous and bounded from below on a \nbh\ of $\bx$;
\item
\label{T5.2}
the collection $\{f_{0},\;\lambda_{k}f_{k}\,(k\in \mathbb{N}\cup\{\infty\})\}$ is uniformly lower semicontinuous on a \nbh\ of $\bx$.
\end{enumerate}
\if{
\AK{30/07/24.
I think it would not be a big loss of generality to assume in (i) the functions $f_{k}\,(k\in \mathbb{N}\cup\{0,\infty\})$ (without $\la_k$) to be \lsc\ and bounded from below on a \nbh\ of $\bx$.
Then the existence of $\la$ could be moved into (ii).

The same simplification can be done also in Theorem~\ref{optimalityb}.
We could go even a bit further and assume the functions to be \lsc\ on the whole space.
Then condition (i) above gets shorter, while in Theorem~\ref{optimalityb} it can be dropped completely.
}
\AH{The problem is the case when some $\la_k$ is zero: In such a case,
$\la_kf_{k}=I_{\dom f_k}$ is not necessarily lower semicontinuous. However, we could do this in Theorem~\ref{optimalityb} for those functions $f_k$ for $k>m$ because we could suppose there that the corresponding $\la_k$'s are positive.
}
}\fi
Then
{conditions \eqref{C5-2} are satisfied, and},
for any $\varepsilon>0$ and $M>0$,
there exist a number
$n\in\N$, and points $x_{k}\in B_{\varepsilon}(\bar{x})$ $(k=0,\ldots,n)$ and $x_{\infty}\in B_{\varepsilon}(\bar{x})$ such that $n>M$ and
\begin{gather}
\notag
f_{0}(x_0)+\sum_{k=1}^n \la_kf_{k}(x_k)+\la_\infty f_\infty(x_\infty)\le f_{0}(\bx),
\\
\label{T5-2}
0\in\partial f_{0}(x_0)+\sum_{k=1}^n{\lambda}_{k}\partial f_{k}(x_{k})+{\lambda}
_{\infty }\partial f_{\infty}(x_\infty)+\varepsilon\B^*
\end{gather}
with the convention
${\lambda}_{k}\partial f_{k}(x_{k})=N_{\dom f_{k}}(x_k)$ when ${\lambda}_{k}=0$ $(k\in \mathbb{N}\cup\{\infty\}).$
\end{theorem}

\begin{proof}
Choose
a solution\ $\lambda =(\lambda_{1},\lambda _{2},\ldots )\in \ell_{+}^{1}$, $\lambda _{\infty }\in\R_+$ of problem $(\ref{problemd})$ satisfying conditions \ref{T5.1} and \ref{T5.2}.
{By Corollary~\ref{C5}\,\ref{C5.1},}
$\bar{x}$ is a (global)
minimum on $X$ of the function
$f_{0}+
\overline{\sum}_{k\in\N}
\lambda_{k}f_{k}+\lambda_{\infty}f_{\infty}$,
and conditions \eqref{C5-2} are satisfied.
The remaining conclusions follow by applying
Theorem~\ref{T3.1} and Remark~\ref{R3.1}.
\qed\end{proof}

\begin{remark}
\label{R6}
Assumption \ref{T5.1} in Theorem~\ref{optimality} can be
rewritten equivalently in the following way:
``\emph{the functions $f_{k}\,(k\in \mathbb{N}\cup\{0,\infty\})$ are lower semicontinuous and bounded from below on a \nbh\ of $\bx$, and the sets $\dom f_{k}$ are closed for all $k\in \mathbb{N}\cup\{\infty\}$ such that $\la_k=0$}''.

The closedness of the domains in the above condition is needed to ensure the lower semicontinuity of the functions $\lambda_{k}f_{k}\,(k\in \mathbb{N}\cup\{\infty\})$ when $\lambda_{k}=0$ for some $k\in \mathbb{N}\cup\{\infty\}$.
Note that, in view of our standing convention $0\cdot (+\infty )=+\infty$, $\lambda_{k}f_{k}$ equals the indicator function of $\dom f_{k}$ when $\lambda_{k}=0$.
{If the sets $\dom f_{k}$ are assumed closed for all $k\in \mathbb{N}\cup\{\infty\}$,}
then the assumption of the existence of a solution of problem $(\ref{problemd})$ can be moved into condition \ref{T5.2}.

If a function is \lsc\ on some set, it is obviously bounded from below on a neighbourhood of any point in this set.
Note that
{condition \ref{T5.1} requires}
the boundedness of all the functions on some \emph{common} neighbourhood of $\bx$.
\end{remark}
\if{
\AK{31/07/24.
Theorems \ref{optimality} and \ref{optimalityb} represent a new (important?) phenomenon in the world of fuzzy optimality conditions, even in finite optimization, which is not observed in Theorem 4.1 and the corresponding results in \cite{FabKruMeh24}: the multipliers are related to the solution $\bx$ and not the `fuzzy' points $x_k$.
This is likely to be specific to the convex case.
This phenomenon should be commented on.
}
\AH{I agree, it would be nice to comment on that; up to my knowledge, it was not observed in the convex setting unless there are a finite number of constraints. I think some similar phenomena should also hold for the non-convex case under semicontinuity/uniform lower firmness ...}
}\fi
\begin{theorem}
\label{optimalityb}
Let the functions $f_k$ $(k\in\N\cup\{0\})$ be convex, and $\bx\in X$ be a solution of problem \eqref{problemp}.
Suppose that
Slater condition \eqref{Sl} holds,
and there exist a number $m\in\N\cup\{0\}$ and a solution\ $\lambda =(\lambda_{1},\lambda _{2},\ldots )\in \ell_{+}^{\infty}$ of problem {\rm (\ref{problemdmas})} such that
\begin{enumerate}
\item
\label{T5.3.1}
the functions $f_0$, $\la_k f_k$ $(k=1,\ldots,m)$, $\la_k f_k^+$ $(k=m+1,\ldots)$ are lower semicontinuous on a \nbh\ of $\bx$;
\item
\label{T5.3.2}
the collection $\{f_{0},\;\lambda_{k}f_{k}\ (k=1,\ldots,m),\;\lambda_{k}f_{k}^{+}\ (k=m+1,\ldots)\}$ is uniformly lower semicontinuous on a \nbh\ of $\bx$.
\end{enumerate}
Then,
{conditions \eqref{C5-4} are satisfied, and},
for any $\varepsilon >0$ and $M>m$, there exist a number $n\in\N$
and points $x_{k}\in B_{\varepsilon }(\bar{x})$ $(k=0,\ldots ,n)$\ such that $n>M$ and
\begin{gather}
\notag
f_{0}(x_{0})
+\sum_{k=1}^m \lambda_{k}f_{k}(x_{k})
+\sum_{k=m+1}^n \lambda_{k}f_{k}^+(x_{k})
\leq f_{0}(\bar{x}),
\\
\label{T7-2}
0\in \partial f_{0}(x_{0})
+\sum_{k=1}^m\lambda
_{k}\partial f_{k}(x_{k})
+\sum_{k=m+1}^n\lambda
_{k}\partial f_{k}^+(x_{k}) +\varepsilon\mathbb{B}^{\ast}
\end{gather}
with the
convention
${\lambda}_{k}\partial f_{k}(x_{k})={\lambda}_{k}\partial f_{k}^+(x_{k})=N_{\dom f_{k}}(x_k)$ when ${\lambda}_{k}=0$ $(k\in\N)$.
\end{theorem}
\if{
\AK{31/07/24.
I think the ``$\eps$-slackness'' conditions can be dropped, and the proof can be shortened.
In fact, the proof is basically the same as that of Theorem~\ref{optimality}.
}
\AH{I removed the ``$\eps$-slackness'' conditions.}
}\fi

\begin{proof}
Since $f_k$ $(k=1,\ldots,m)$ are \lsc, $f_k^+\ge0$ $(k=m+1,\ldots)$ and $\lambda\in \ell_{+}^{\infty}$, the functions $f_0$, $\la_k f_k$ $(k=1,\ldots,m)$, $\la_k f_k^+$ $(k=m+1,\ldots)$ are bounded from below on a \nbh\ of $\bx$.
Choose
a solution\ $\lambda =(\lambda_{1},\lambda _{2},\ldots )\in \ell_{+}^{\infty}$ of problem \eqref{problemdmas} satisfying conditions \ref{T5.3.1} and \ref{T5.3.2}.
By Corollary~\ref{C5}\,\ref{C5.2},
$\bar{x}$ is a (global)
minimum on $X$ of the function
$f_{0}+{\sum}_{k=1}^{m}
\lambda_{k}f_{k}+{\sum}_{k=m+1}^{+\infty}
\lambda_{k}f_{k}^+$
and conditions \eqref{C5-4} are satisfied.
The remaining conclusions follow by applying
Theorem~\ref{T3.1} and Remark~\ref{R3.1}.
\qed\end{proof}

\begin{remark}
\begin{enumerate}
\item
The observations regarding Theorem \ref{optimality} made in Remark~\ref{R6} are largely applicable to Theorem \ref{optimalityb}.
In particular, assumption \ref{T5.3.1} in Theorem~\ref{optimalityb} can be
rewritten equivalently in the following way:
``\emph{the functions $f_{k}$ $(k=1,\ldots,m)$, $f_{k}^+$ $(k=m+1,\ldots)$  are lower semicontinuous on a \nbh\ of $\bx$, and the sets $\dom f_{k}$ are closed for all $k\in \mathbb{N}\cup\{\infty\}$ such that $\la_k=0$}''.

\if{
\AK{21/07/24.
I have temporarily restored the ``positive part'' functions and hidden computing their subdifferentials to comply with the current ``constructive part'' of Theorem~\ref{tmain} which makes all the multipliers positive, also for non-active constraints.\\
{\color{blue} We have added the estimates on the changes in the values of the functions.}
}}\fi

\item
Theorems \ref{optimality} and \ref{optimalityb} give fuzzy multiplier rules \eqref{T5-2} and \eqref{T7-2}: the subdifferentials are computed at some points $x_0,x_1,\ldots,x_n$ (and also $x_\infty$ in the case of \eqref{T5-2}) close (up to $\eps$) to the solution $\bx$.
At the same time, the corresponding multipliers $\la_0,\la_1,\ldots,\la_n$ (and $\la_\infty$ in the case of \eqref{T5-2}) are determined not in a fuzzy way: they are parts of the solutions of the corresponding dual problems and do not depend on $\eps$.
\item
An alternative way to close the duality gap in convex
infinite programming problems consists of
considering the so-called \emph{characteristic cone}
\begin{equation*}
K:=\cone\co\bigcup\limits_{k\in\N}\epi
f_{k}^{\ast }.
\end{equation*}
associated with the
constraint system $\mathcal{S}:=\{f_{k}(x)\leq 0\;
(k\in\N)\}$.
An extension of the Farkas lemma (see, e.g., \cite[Theorem 8.1.5]{CorHanLop23})
establishes that the inequality $\left\langle a,x\right\rangle \leq \alpha $
is consequence of $\mathcal{S}$ if and only if
$\left( a^{\ast },\alpha \right) \in \cl K$.
If the characteristic cone $K$ is
weak$^*$-closed
(\emph{Farkas-Minkowski property}), then, according to
\cite[Theorem~8.2.12]{CorHanLop23},
there is no
duality gap and strong duality holds. Therefore, if we add to the original
constraint system those affine inequalities $\left\langle a,x\right\rangle
\leq \alpha $ such that $\left( a^{\ast },\alpha \right) \in (\cl %
K)\diagdown K$ we
close the duality gap, but in the dual pair with an
extended system of primal constraints. The advantage of the methodology we
propose in Theorems \ref{optimality} and \ref{optimalityb}
is that we only need to add
the well-characterized function
$f_{\infty }:=\limsup_{k\rightarrow \infty }f_{k}$, and appeal to the functions $f_k^+$ for large $k$, respectively.
\end{enumerate}
\end{remark}

The analysis above is illustrated by the following example.

\begin{example}
\if{
\PM{
``Why not $u_1=x_1$ and $u_2=x_2$?''}
\AK{12/05/25.
Because, $x_1,x_2,\ldots$ are used to denote points in $\R^2$.}
}\fi
We reconsider the problem in Example~\ref{E1.1}.
Observe that the functions $f_k$ $(k=0,1,\ldots)$ and $f:=\sup_{k\in\N}f_k$ are finite, convex and continuous, $f_{\infty}(x):=\limsup_{k\rightarrow\infty}f_k(x)=-u_2$, where $x=(u_1,u_2)\in\R^2$, and Slater condition \eqref{Sl} holds.

With $\lambda=(\la_1,\la_2,\ldots)\in
{\ell^{1}_+}$ and $\la_\infty\in
{\R_+}$, the Lagrangian function \eqref{Lag} takes the form:
\begin{equation*}
L(x,\lambda ,\lambda _{\infty })=\Big( \lambda_{2}+\sum_{k=3}^{+\infty}\frac{\lambda _{k}}{k}\Big) u_{1}+\Big(1-\lambda_1-\sum_{k=3}^{+\infty}\lambda_{k}-\lambda _{\infty }\Big)u_{2}-\lambda_1.
\end{equation*}
(Compare with the function \eqref{E1.1-2} in Example~\ref{E1.1}.)
Hence, $\inf_{x\in X}L(x,\lambda ,\lambda _{\infty })$ equals $-\lambda_1$ if $\lambda_2+\sum_{k=
3}^{+\infty}\frac{\lambda _{k}}{k}=1-\lambda_1-\sum_{k= 3}^{+\infty}\lambda_{k}-\lambda _{\infty}=0$,
{i.e., $\la_k=0$ for $k=2,3,\ldots$ and $\la_1+\la_\infty=1$,}
and $-\infty$ otherwise.
Thus, the corresponding dual problem \eqref{problemd} looks like this:
\begin{equation*}
\begin{array}{ll}
\text{maximize} & -\lambda_1\\
\text{subject to} & \lambda_1+\lambda_{\infty}=1,\\ &\lambda_1\ge0,\;\lambda_{\infty}\ge0.
\end{array}
\end{equation*}
Its solution is obviously $\{0_{\ell ^{1}},1\},$ and the optimal value equals zero, coinciding with that of the primal problem.
This confirms the strong duality
{between problems \eqref{problemp} and \eqref{problemd}} stated in Theorem~\ref{tmain}.

{We now check the strong duality between problems \eqref{problemp} and \eqref{problemdmas}.
Similarly to the above, with $\lambda=(\la_1,\la_2,\ldots)\in\ell^{\infty}_+$ and any $m>2$, the Lagrangian function \eqref{Lag2} takes the form:
\begin{align*}
L_{m}(x,\lambda)=\Big( \lambda_2+\sum_{k=3}^{m}\frac{\lambda_{k}}{k}\Big) u_{1}+\Big(1-\la_1-\sum_{k=3}^{m}\lambda_{k}\Big)u_{2} +\sum_{k=m+1}^{+\infty}\lambda_{k}\Big(\frac1k{u_1}-u_2\Big)^+ -\la_1.
\end{align*}
}
{We are going to show that
\begin{gather}
\label{E4.1-4}
\inf_{x\in X}L_{m}(x,\lambda)=
\begin{cases}
-\lambda_1 & \text{if } \lambda_1=1, \text{ or } \lambda_1<1 \text{ and } \lambda_1 +\sum_{k=j}^{\infty} \lambda_{k}>1 \text{ for all } j>m,\\
-\infty & \text{otherwise.}
\end{cases}
\end{gather}
Then $v(\mathtt{D}_m) =\sup_{0\le\la_1\le1
} (-\la_1)=0$, confirming the strong duality
between problems \eqref{problemp} and
\eqref{problemdmas}.
We consider several cases.}

If $\max\{\lambda_2,\ldots,\lambda_{m}\}>0$, then, setting $x_t:=(-t,0)$ for $t>0$, we have $L_{m}(x_t,\lambda) =-\big(\lambda_2+\sum_{k=3}^{m} \frac{\lambda_{k}}{k}\big)t-\la_1$,
and consequently,
\begin{align*}
\inf_{x\in X}L_{m}(x,\lambda)\le-\lim_{t\to+\infty} \Big(\lambda_2+\sum_{k=3}^{m} \frac{\lambda_{k}}{k}\Big)t-\la_1=-\infty.
\end{align*}
From now on, we assume that $\lambda_2=\ldots=\lambda_{m}=0$.
The Lagrangian function takes a simpler form:
\begin{align*}
L_{m}(x,\lambda)&=(1-\lambda_1)u_{2} +\sum_{k=m+1}^{+\infty} \lambda_{k}\Big(\frac1k{u_1}-u_2\Big)^+-\la_1.
\end{align*}

{If $\la_1>1$, then, setting $x_t:=(0,t)$ for $t>0$, we have $L_{m}(x_t,\lambda)=(1-\lambda_1)t-\la_1$, and consequently,
\begin{align*}
\inf_{x\in X}L_{m}(x,\lambda)\le\lim_{t\to+\infty} (1-\lambda_1)t-\la_1=-\infty.
\end{align*}
}

{If $\la_1=1$, then clearly $\inf_{x\in X}L_{m}(x,\lambda)=-\la_1=-1$.}

{Let $0\le\la_1<1$.
For each $j\in\N$, denote $\bar\la_j:=\sum_{k=j}^{+\infty}\lambda_{k}$, and observe that the sequence of nonnegative (extended) numbers $\bar\la_1,\bar\la_2,\ldots$ is nonincreasing, and consequently, converges to some~$\bar\la$.}

{If $\la_1+\bar\la_j<1$ for some $j>m$, then setting $x_t:=(-(j-1)t,-t)$ for $t>0$, we have
\begin{align*}
L_{m}(x,\lambda)&=(\lambda_1-1)t +\sum_{k=j}^{+\infty} \lambda_{k}\Big(1-\frac{j-1}k\Big)t-\la_1 \le(\la_1+\bar\la_j-1)t-\la_1,
\end{align*}
and consequently,
\begin{align*}
\inf_{x\in X}L_{m}(x,\lambda)\le\lim_{t\to+\infty} (\la_1+\bar\la_j-1)t-\la_1 =-\infty.
\end{align*}
Let $\la_1+\bar\la_j\ge1$ for all $j>m$.
Then $\la_1+\bar\la\ge1$ and $\bar\la>0$.}

{If $\la_1+\bar\la_{j_0}=1$ for some $j_0>m$, then $\la_1+\bar\la_j=1$ for all $j\ge j_0$, and consequently, $\la_j=0$ for all $j\ge j_0$, yielding $\bar\la=\bar\la_{j_0}=0$; a conradiction.
So, this case is impossible.}

{Let $\la_1+\bar\la_j>1$ for all $j>m$.
If $u_2\ge0$, then, obviously, $L_{m}(x,\lambda)\ge-\la_1$.
Let $\al\in(0,1)$.
For any $u_2<0$ and $u_1\in\R$, choose an integer $j>\max\big\{m,\frac{u_1}{u_2(1-\al)}\big\}$.
(Note that the second component can only be useful when $u_1<0$.)
Then $\frac{u_1}{ku_2}<1-\al$ for all $k>j$, and consequently,
\begin{align*}
L_{m}(x,\lambda)&\ge(1-\lambda_1)u_{2} +\sum_{k=j}^{+\infty} \lambda_{k}\Big(\frac{u_1}{ku_2}-1\Big)u_2-\la_1 \\&\ge(1-\lambda_1-\al\bar\lambda_j)u_2-\la_1
\ge(1-\lambda_1-\al\bar\lambda)u_2-\la_1.
\end{align*}
Passing to the limit as $\al\uparrow1$, we obtain $L_{m}(x,\lambda)\ge(1-\lambda_1-\bar\lambda)u_2-\la_1 \ge-\la_1$.
Taking into account that $L_{m}(x,\lambda)=-\la_1$ when $u_1=u_2=0$, we conclude that $\inf_{x\in X}L_{m}(x,\lambda)=-\la_1$}.
This completes the proof of \eqref{E4.1-4}
and, as a consequence, the strong duality
between problems \eqref{problemp} and \eqref{problemdmas}.

For the solution
$\bar{x}:=(0,0)$
{of problem \eqref{problemp}},
we have $f_k(\bx)=0$, $k=0,2,3,\ldots$, and $f_1(\bx)=-1$.
For the
solution $\{(0_{\ell ^{1}},1)\}$
{of problem \eqref{problemd}},
the corresponding family $\{f_0,\ \lambda _{k}f_{k}\ (k\in\N),\ \lambda_\infty f_\infty\}=\{f_0,\ 0,0,\ldots,\ f_\infty\}$
is uniformly lower semicontinuous on a \nbh\ of~$\bx$.
Thus, the conditions of Theorem \ref{optimality}
are satisfied.
The conclusions hold true as well (even with $\eps=0$).
Indeed, one can take $x_0=x_1=\ldots=x_\infty:=\bx$.
The subdifferentials are easily computed:
$\sd f_0(x_0)=\Big(
\begin{array}{c}
0 \\
1
\end{array}
\Big)$,
$\sd f_\infty(x_\infty)=\Big(
\begin{array}{c}
0 \\
-1
\end{array}
\Big)$.
Hence,
$$\Big(\begin{array}{c}
0 \\
0
\end{array}
\Big)\in\sd f_0(x_0)+\sum_{k=1}^n{\lambda}_{k}\partial f_{k}(x_{k})+\bar\la\sd f_\infty(x_\infty).$$

{Theorem \ref{optimalityb} is also applicable in this example.}
\end{example}
\if{
\AH{I would propose to add the following (general) references:\\
N. Dinh, M.A. Goberna, M.A. LÃ³pez, and M. Volle. Relaxed Lagrangian duality in convex infinite optimization: reverse strong duality and optimality. J. Appl. Numer. Optim. 4:3--18, 2022.\\
M.A. Goberna and M. Volle. Duality for convex infinite optimization on linear spaces. Optim. Lett. 16:2501-2510. 2022.\\
A. Hantoute, R. Correa. Integral-type representations for the subdifferential of
suprema, submitted, 2023.}
\AK{12/08/24.
Added.
Check the red pieces on pages 4 and 20.}
\fi
\if{
\begin{remark}
Theorem \ref{optimalityb} is not applicable in the above example.
Indeed, by Theorem \ref{tmain}, since $\{0_{\ell ^{1}},1\}$ is a solution of \eqref{problemd}, the stationary sequence $(2,2,\ldots)\in\ell^\infty$ is a solution of \eqref{problemdmas}.
We can easily see that the collection of functions
\blue{$\{f_0,\ 2 f_k\  (k\leq m),\ 2 f_k^+\  (k\geq m+1)\}$ is not uniformly lower semicontinuous}. Indeed,  for each $n\geq 1$ and $\delta>0$ we have that
\begin{gather*}
\inf_{\rm{diam}(x_0,\ldots,x_n)\leq \delta} \big(f_0(x_0)+2\sum\nolimits_{k=1}^m f_k(x_k)
+2\sum\nolimits_{k=m+1}^n f_k^+(x_k)\big)
\\
\leq \inf_{x\in X} \big(f_0(x)+2\sum\nolimits_{k=1}^m f_k(x)
+2\sum\nolimits_{k=m+1}^n f_k^+(x)\big)
\leq  -1.
\end{gather*}
\AK{24/07/24.
Take $m=2$.
Then $\hat\la=(0,0,1,1,\ldots)$.
$$\liminf\limits_{\substack{\diam(u_{13},\ldots,u_{1n})\to0\\ \diam(u_{20},u_{23}\ldots,u_{2n})\to0,\, (u_{1k},u_{2k})\in U}} \Big(u_{20}+\sum_{k=3}^n(\frac1ku_{1k}-u_{2k})_+\Big) =0$$ when $n$ is large enough.
Here $U$ is a neighborhood of $(0,0)$.
}
\AH{Perhaps I am wrong, but following my understanding of the uniform lower semicontinuity, the last liminf is not larger than $-\delta$, where we take $U:=[-\delta,-\delta]^2$ is any fixed neighborhood of $\bx=(0,0).$ Indeed, for any $u_1$, $u_2\in U$ and $n \geq 3$ we have
\begin{gather*}
\liminf\limits_{\substack{\diam(u_{13},\ldots,u_{1n})\to0\\ \diam(u_{20},u_{23}\ldots,u_{2n})\to0,\, (u_{1k},u_{2k})\in U}} \Big(u_{20}+\sum_{k=3}^n(\frac{1}{k}u_{1k}-u_{2k})_+\Big) \\
\le
\liminf\limits_{\substack{\diam(u_{13},\ldots,u_{1n})\to0,\ u_{13}=\ldots=u_{1n}=u_1\\ \diam(u_{20},u_{23}\ldots,u_{2n})\to0,\, u_{20}=u_{23}=\ldots=u_{2n}=u_2,\ (u_{1k},u_{2k})\in U}} \Big(u_{20}+\sum_{k=3}^n(\frac{1}{k}u_{1k}-u_{2k})_+\Big)\\
=\inf_{(u_1,u_2)\in U}\Big(u_{2}+\sum_{k=3}^n(\frac{1}{k}u_{1}-u_{2})^+\Big)
\\
\leq \inf_{(u_1,u_2)\in U,\ \frac{1}{k}u_{1}-u_{2}\leq 0 \, (3\le k\le n)}\Big(u_{2}+\sum_{k=3}^n(\frac{1}{k}u_{1}-u_{2})^+\Big)\\
= \inf_{(u_1,u_2)\in U,\ \frac{1}{k}u_{1}-u_{2}\leq 0 \, (3\le k\le n)} u_2
\le -\delta\quad (\text{take for instance } u_1=0 \text{ and } u_2=-\delta).
\end{gather*}}

\AK{31/07/24.
In the last estimate, you cannot take $u_2=-\delta$.
It must be $u_2\ge u_1/n\ge-\de/n$, and then you need to let $n\to+\infty$.}
\if{
\AK{24/07/24.
This seems to illustrate a gap in our theory around Theorem \ref{optimalityb}.
$-1$ above is caused by the non-active constraint with $k=2$.
Something should be said about the choice of $m$.
The cases $m=1$ and $m>1$ seem to be different.
In any case, I would expect to have $\hat\la_2=0$.
This would make the above infimum equal 0.
Can we make Theorem \ref{optimalityb} work in this example?
This would add some value to our theory.}
}\fi
{\color{olive}So,
\begin{gather*}
\liminf_{n\to\infty}\liminf_{\rm{diam}(x_0,\ldots,x_n)\to 0} \big(f_0(x_0)+2\sum\nolimits_{k=1}^m f_k(x_k)
+2\sum\nolimits_{k=m+1}^n f_k^+(x_k)\big)\\
\leq -1<0=v(\mathtt{P}).
\end{gather*}
 Despite this, the conclusions of Theorem \ref{optimalityb} is satisfied too, because the uniform lower semicontinuity is not a necessary condition.} For instance,
for any $\eps>0$ and $n>3$, one can take
$x_0=x_1=x_2=
\Big(\begin{array}{c}
0 \\
0
\end{array}
\Big)$,
$\la_1=\la_2=0$,
$x_k=
\Big(\begin{array}{c}
0 \\
-\eps/2
\end{array}
\Big)$ and $\la_k=\frac1{n-2}$,
$k=3,\ldots,n$.
Then
$$
\sd f_0(x_0)+ \sum\nolimits_{k=3}^n\lambda_{k}\partial f_{k}(x_{k})=
\Big(\begin{array}{c}
0 \\
1
\end{array}
\Big)
+\frac1{n-2}\sum\nolimits_{k=3}^n
\Big(\begin{array}{c}
\frac1k \\
-1
\end{array}
\Big)
=
\Big(\begin{array}{c}
\frac1{n-2}\sum\nolimits_{k=3}^n\frac1k \\
0
\end{array}
\Big)
$$ and $\frac1{n-2}\sum\nolimits_{k=3}^n\frac1k<\eps$  if $n$ is large enough.
\sloppy
\end{remark}
}\fi
\if{
\AK{20/07/24.
It seems to me that the collection IS uniformly lower semicontinuous.
The conclusions of Theorem \ref{optimalityb} seem to be satisfied too.
For any $\eps>0$ and $n>3$, one can take
$x_0=x_1=x_2=
\Big(\begin{array}{c}
0 \\
0
\end{array}
\Big)$,
$\la_1=\la_2=0$,
$x_k=
\Big(\begin{array}{c}
0 \\
-\eps/2
\end{array}
\Big)$ and $\la_k=\frac1{n-2}$,
$k=3,\ldots,n$.
Then $\sd f(x_0)+ \sum\nolimits_{k=3}^n\lambda_{k}\partial f_{k}(x_{k})=
\Big(\begin{array}{c}
0 \\
1
\end{array}
\Big)
+\frac1{n-2}\sum\nolimits_{k=3}^n
\Big(\begin{array}{c}
\frac1k \\
-1
\end{array}
\Big)
=
\Big(\begin{array}{c}
\frac1{n-2}\sum\nolimits_{k=3}^n\frac1k \\
0
\end{array}
\Big)
$ and $\frac1{n-2}\sum\nolimits_{k=3}^n\frac1k<\eps$  if $n$ is large enough.
Am I missing something?
}
}\fi

\section{Conclusions}\label{sec:conclusions}

In this work, we have explored new duality perspectives for infinite optimization problems with infinitely many constraints in a Banach space. Classical duality schemes based on Haar duality often fail to ensure strong duality. To address this issue, we have 
studied broader Lagrangian-type dual formulations that incorporate the concept of infinite sum for arbitrary collections of functions introduced in \cite{HanKruLop1}. We have proposed two new dual problems: the first one, denoted (\ref{problemd}), is constructed directly from the original data of problem (\ref{problemp}); the second involves a Lagrangian that includes both the infinite sum of the constraint functions and a limit function, reflecting the role of compactification techniques used in our approach. 
We have shown that under the standard Slater condition these dual formulations are better suited to guarantee zero duality gap and strong duality (existence of dual solutions). Furthermore, we have established general optimality conditions (in the form of multiplier rules) for problem (\ref{problemp}) by applying a subdifferential calculus rule from \cite{HanKruLop1}.



\section*{Acknowledgements}

The authors have benefited from fruitful discussions with Pedro P\'erez-Aros
and wish to express him their gratitude.

This work was supported by Grant PID2022-136399NB-C21 funded by MICIU/AEI/10.13039/501100011033 and by ERDF/EU, MICIU of Spain and Universidad de Alicante (Contract Beatriz Galindo BEA-GAL 18/00205), and by AICO/2021/165 of Generalitat Valenciana, and by
Programa Propio Universidad de Alicante (INVB24-01 - Ayudas para Estancias de Personal Investigador Invitado 2025).

Parts of the work were done during Alexander Kruger's stays at the University of Alicante and the Vietnam Institute for Advanced Study in Mathematics in Hanoi.
He is grateful to both institutions for their hospitality and supportive environment.

\section*{Declarations}

\noindent{\bf Data availability. }
Data sharing is not applicable to this article as no datasets have been generated or analysed during the current study.

\noindent{\bf Conflict of interest.} The authors have no competing interests to declare that are relevant to the content of this article.

\addcontentsline{toc}{section}{References}

\bibliographystyle{spmpsci}
\bibliography{buch-kr,kruger,kr-tmp}

\end{document}